\pdfoutput=1
\documentclass[11pt,a4paper,oneside]{amsart}

\usepackage{pdfpages,amsfonts,amsthm,amssymb,latexsym,ucs,graphicx,leftidx,fancyhdr}
\usepackage[left=2.5cm,right=2.5cm,bottom=2cm,top=2cm]{geometry} 

\usepackage{color,amssymb,latexsym,amsfonts,textcomp,fancyhdr,calc,graphicx}
\usepackage{pdfpages,amsfonts,amsthm,amssymb,latexsym,graphicx,leftidx,fancyhdr}
\usepackage{amsmath,amstext,amsthm,amssymb,amsxtra}

\usepackage[colorlinks,citecolor=red,hypertexnames=false]{hyperref}
\usepackage{dsfont}
\usepackage[framemethod=tikz]{mdframed}
\usepackage{asymptote}

\usepackage[ULforem]{ulem}
\usepackage{cancel}
\usepackage{cite}

\usepackage{txfonts,pxfonts,tikz} 
\usepackage{tgschola}
\usepackage[T1]{fontenc}

\newtheorem{theorem}{Theorem}
\newtheorem{corollary}[theorem]{Corollary}
\newtheorem{lemma}[theorem]{Lemma}

\newtheorem{remark}{Remark}

\newtheorem{definition}{Definition}[section]
\theoremstyle{definition}

\newcommand{\beql}[1]{\begin{equation}\label{#1}}
	\newcommand{\eeq}{\end{equation}}
\newcommand{\comment}[1]{}
\newcommand{\Ds}{\displaystyle}

\newcommand{\Abs}[1]{{\left|{#1}\right|}}

\newcommand{\Floor}[1]{{\left\lfloor{#1}\right\rfloor}}

\newcommand{\Set}[1]{{\left\{{#1}\right\}}}

\newcommand{\RR}{{\mathbb R}}

\newcommand{\CC}{{\mathbb C}}
\newcommand{\ZZ}{{\mathbb Z}}
\newcommand{\NN}{{\mathbb N}}
\newcommand{\TT}{{\mathbb T}}

\newcommand{\one}{{\bf 1}}

\newcommand{\ft}[1]{\widehat{#1}}

\newcounter{rem}
\newcounter{step}
\newcounter{mysec}
\newcounter{mysubsec}[mysec]
\begin{document}
	
	\title{How many Fourier coefficients are needed?}
	
	\author{Benedikt Diederichs}
	\address{Institute of Biological and Medical Imaging, Helmholtz Zentrum München, 85764 Neuherberg, Germany}
	\email{benedikt.diederichs@helmholtz-muenchen.de}
	
	\author{Mihail N. Kolountzakis}
	\address{Department of Mathematics and Applied Mathematics, University of Crete, Voutes Campus, 70013 Heraklion, Crete, Greece.}
	\email{kolount@gmail.com}
	
	\author{Effie Papageorgiou}
	\address{Department of Mathematics and Applied Mathematics, University of Crete, Voutes Campus, 70013 Heraklion, Crete, Greece.}
	\email{papageoeffie@gmail.com}

	\thanks{Supported by the Hellenic Foundation for Research and Innovation, Project HFRI-FM17-1733 and by grant No 4725 of the University of Crete.}
	
	\begin{abstract}
		We are looking at families of functions or measures on the torus which are specified by a finite number of parameters $N$. The task, for a given family, is to look at a small number of Fourier coefficients of the object, at a set of locations that is predetermined and may depend only on $N$, and determine the object. We look at (a) the indicator functions of at most $N$ intervals of the torus and (b) at sums of at most $N$ complex point masses on the multidimensional torus. In the first case we reprove a theorem of Courtney which says that the Fourier coefficients at the locations $0, 1, \ldots, N$ are sufficient to determine the function (the intervals). In the second case we produce a set of locations of size $O(N \log^{d-1} N)$ which suffices to determine the measure.
	\end{abstract}
	
	\keywords{Interpolation, sparse exponential sums, non-harmonic exponential sums, Fourier coefficients, inverse problem}
	
	\makeatletter
	\@namedef{subjclassname@2020}{\textup{2020} Mathematics Subject Classification}
	\makeatother
	\subjclass[2020]{41A05, 41A27, 42A15, 42A16}
	
	\maketitle
	
	\tableofcontents
	
	\setlength{\parskip}{0.5em}

	\section{Introduction}
	
	Assume that the function $f$ belongs to a given $k$-parameter explicit family of functions.
	Can we recover $f$ by looking at $k$ (or, at least, not many more than $k$) of its Fourier coefficients?
	This situation is often called ``a signal with a finite rate of innovation'' in the engineering literature \cite{vetterli2002sampling}.
	We are particularly interested in families where the dependence on the parameters is non-linear.
	
	The recovery of a function by not too many of its Fourier coefficients may be viewed as a problem in the general field of sparse representation (see, e.g., \cite{candes2006robust}). The specific rules that apply to this paper are the following:
	\begin{itemize}
		\item We consider functions or measures on the torus $\TT=\RR/\ZZ$ or $\TT^d$, $d\in\NN$. The classes of functions we examine have a finite number of degrees of freedom whose number is constrained by the parameter $N$. For instance we might consider sums of point masses on $\TT$ with the number of points being at most $N$.
		\item We seek an a priori known finite set $\Omega = \Omega_N$ of Fourier coefficients (a subset of $\ZZ$ or $\ZZ^d$) which are assumed to be known for our class of functions. The set $\Omega$ is allowed to depend on $N$ and on nothing else.
		\item The aim is to show that the mapping $f \to \ft{f}\restriction_\Omega$ is one to one. Though our proofs can often be turned into algorithms for the recovery of $f$ we do not concern ourselves with matters of numerical stability or efficiency.
	\end{itemize}
	
We deal with two problems in this paper.
\begin{enumerate}
\item {\bf Intervals in $\TT$.}

Our function $f$ is the \textit{indicator} function of the union of at most $N$ open intervals on $\TT$. Courtney \cite{courtney2010unions} has shown that such a function is determined by its Fourier coefficients at the locations $0, 1, \ldots, N$. We give a new proof of this fact which is completely elementary (Courtney's proof uses Blaschke products and conformal mapping). We further discuss the problem of whether functions of this class are determined by different sets of Fourier coefficients.

We need to emphasize here that, apart from Courtney \cite{courtney2010unions} this problem has not been considered elsewhere. There are papers \cite{vetterli2002sampling} where one recovers a function on $\TT$ which is piecewise constant (or even piecewise a polynomial) with the correct number of samples (roughly equal to the number of degrees of freedom) but all these methods fail to take into account the fact that the function only takes two values (0 or 1) and will accordingly use a number of samples that is larger than the minimum by at least ${\rm const.} N$ samples. The ``extreme'' nonlinearity of this problem (not only in allowing variable nodes in the decomposition of $\TT$, but also in the values of the function) does not seem to make it amenable to the usual methods such as Prony's method, if one wants to use the minimal number of samples (or close to the minimal).

This we do in \S\ref{sec:intervals}.

\item {\bf Point masses in $\TT^d$.}

We examine the class of measures which are sums of at most $N$ complex point masses on $\TT^d$. Using the corresponding question in dimension 1 (solvable with the so-called Prony's method) we show an explicit set of locations $\Omega$ such that the Fourier coefficients on $\Omega$ determine the measure. This set $\Omega$ is of size $O(N \log^{d-1} N)$. We believe it is the first such set given of size $o(N^2)$, though several other methods have been described for this problem under additional assumptions on the locations of the point masses \cite{plonka2013many,maravic2004exact,kunis2016multivariate,diederichs2015parameter,potts2013parameter,cuyt2018multivariate,diederichs2017projection,sauer2017prony, sauer2018prony}. We emphasize that the set $\Omega$ depends only on $N$ and is not determined on the fly by looking at the Fourier coefficients of the measure.

This we show in \S\ref{sec:2d} where we also give a set $\Omega$ of size $O(k^{d-1} N)$ when we assume that the set of point masses has at most $k$ points with the same $x$-coordinate.

In \S\ref{sec:2d} we also describe a general connection of this problem with the problem of interpolation.
\end{enumerate}

	\section{At most $N$ intervals on $\TT$}\label{sec:intervals}
	
	\subsection{Determination from the Fourier coefficients at $0, 1, \ldots, N$}\label{sec:courtney}
	
	We consider sets $E \subset \TT$ of the form
	$$
	E = \bigcup_{j=1}^k (a_j, b_j)
	$$
	where $k \le N$ and the open intervals $(a_j, b_j)$ are disjoint.
	We show that $f = \one_E$ is determined by the complex data
	$$
	\ft{f}(0), \ft{f}(1),\ldots,\ft{f}(N).
	$$
	The family has $\le 2N$ real degrees of freedom and the data has $2N+1$, since $\ft{f}(0)$ is always real.
	
	Several similar problems with functions supported on intervals are treated in \cite{plonka2013many}.

	\begin{theorem}\label{th:intervals}
		Suppose that the sets $E,E' \subseteq \TT$ are both unions of at most $N$
		open arcs and that $\ft{\chi_E}(\nu) = \ft{\chi_{E'}}(\nu)$ for
		$\nu=0,1,\ldots,N$.
		Then $E = E'$. 
	\end{theorem}

	\begin{proof}
		For $x=(x_1,\ldots,x_n)$ let $\sigma_k(x)$ denote the $k$-th elementary
		symmetric function of
		the variables $x_i$ and let $s_k(x) = \sum_{j=1}^n x_j^k$ denote the $k$-th
		power sum of the $x_i$.
		
		We use the Newton-Girard formulas
		\begin{equation}\label{NG}
			k \cdot \sigma_k(x) = \sum_{i=1}^k (-1)^{i-1} \sigma_{k-i}(x) s_i(x),\ \ (k \ge 1).
		\end{equation}
		Note that $\sigma_0(x)=1$.
		What is important about these formulas is that if we know $s_1,\ldots,s_\nu$
		then we know also the numbers $\sigma_1,\ldots,\sigma_\nu$, for all $\nu\ge1$.
		The precise dependence is irrelevant for our purposes.

		Suppose $N$ is given and that the sets $E = \bigcup_{j=1}^n I_j$ and $E' = \bigcup_{j=1}^{n'} I_j'$ (with $n, n' \le N$ and the $I_j, I_j'$ being arcs) have the same Fourier coefficients of order up to $N$:
		$$
		\ft{\chi_E}(\nu) = \ft{\chi_{E'}}(\nu),\ \ \nu=0,1,2,\ldots,N.
		$$
		If $I_j=(a_j, b_j)$ and $I_j'=(a_j', b_j')$ then, differentiating the functions
		$\chi_E, \chi_{E'}$, we obtain
		that the measures
		$$
		\mu = \sum_{j=1}^n \delta_{a_j} - \delta_{b_j},\ \ 
		\mu' = \sum_{j=1}^{n'} \delta_{a_j'} - \delta_{b_j'},
		$$
		have the same Fourier coefficients of order up to $N$. Writing $z_j = e^{-2\pi i a_j}, w_j = e^{-2\pi i b_j},
		z_j' = e^{-2\pi i a_j'}, w_j' = e^{-2\pi i b_j'}$ we obtain the relations 
		$$
		\sum_{j=1}^n z_j^\nu - w_j^\nu = \sum_{j=1}^{n'} z_j'^{\nu}-w_j'^{\nu},\ \ \ \nu=1,2,\ldots,N.
		$$
		From this we get $s_\nu(z,w') = s_\nu(z',w)$, for $\nu=1,2,\ldots,N$, where
		$$
		s_\nu(z,w') := \sum_{j=1}^n z_j^\nu + \sum_{j=1}^{n'} w_j'^\nu,\mbox{\ and\ }
		s_\nu(z',w) = \sum_{j=1}^{n'}z_j'^{\nu} + \sum_{j=1}^n w_j^\nu.
		$$
		By the Newton-Girard formulas the numbers $s_1,\ldots,s_N$ determine the numbers $\sigma_1,\ldots,\sigma_N$,
		hence we have
		\begin{equation}\label{sigmaeq}
			\sigma_\nu(z,w') = \sigma_\nu(z',w),\ \ \ \nu=1,2,\ldots,N.
		\end{equation}
		Write $M=n+n'\le 2N$ and observe that
		$$
		\sigma_M(z,w') = \sigma_M(z',w),
		$$
		i.e.,
		$\prod z_j w_j^{-1} = \prod z_j' w_j'^{-1}$. This comes from
		the fact that the total length of $E$ and $E'$ is the same, as testified by
		$\ft{\chi_E}(0) = \ft{\chi_{E'}}(0)$.
		
		We now use the fact that $\Abs{z_j}=\Abs{w_j}=\Abs{z_j'}=\Abs{w_j'}=1$:
		\begin{equation}\label{connect}
			\overline{\sigma_k(z,w')} = \sigma_k\left(\frac{1}{z},\frac{1}{w'}\right) =
			\frac{\sigma_{M-k}(z,w')}{\sigma_M(z,w')},\ \ \ k=0,1,\ldots,M,
		\end{equation}
		and similarly for the elementary symmetric functions of the vector $(z', w)$.
		For $k=1,2,\ldots,M-N-1\le N$ we obtain from \eqref{connect} the missing values
		of $\sigma_\nu$ for $\nu=N+1,\ldots,M-1$.
		
		We have proved that
		$$
		\sigma_\nu(z,w') = \sigma_\nu(z',w),\ \ \ \nu=0,1,2,\ldots,M,
		$$
		hence the multisets $\Set{z_j, w_j'}$ and $\Set{z_j', w_j}$ are equal,
		since the elementary symmetric functions determine the polynomials $p$ and
		$q$ with roots at $\Set{z,w'}$ and $\Set{z',w}$ respectively and they are equal.
		But $\Set{z_j} \cap \Set{w_j} = \Set{z_j'} \cap \Set{w_j'} = \emptyset$ so
		the only possibility is that $\Set{z_j} = \Set{z_j'}$ and $\Set{w_j} =
		\Set{w_j'}$, as we had to show.
		
	\end{proof}
	
	\begin{remark}
	Of course it is also possible to solve this problem using Prony's method, which we will introduce in Section \ref{sectonetorus}. However, Prony's method will require more samples, as it cannot exploit the fact that the coefficients of $\hat{\chi_E}$ alternate between plus and minus one.
	\end{remark}
	
	\subsection{Determination from other sets of Fourier coefficients}
	
	The problem is sensitive to the choice of which Fourier coefficients to use in order to determine the set, even in the case of one interval $E=(a,b)$. In this case, we have 
	\begin{equation}\label{onedim}
		2\pi i\nu\ft{\chi_E}(\nu)=e^{-2\pi i\nu a}-e^{-2\pi i\nu b}:=z^{\nu}-w^{\nu},\; \nu\neq 0.
	\end{equation}
	(Here, again, $z = e^{-2\pi i a}$, $w = e^{-2\pi i b}$.)

	A single Fourier coefficient is not enough to determine the interval uniquely: for $\nu=0$ this is obvious as $\ft{\chi_E}(0)=b-a$. For $\nu=1$, consider $E'=(a',b')$ such that $z'=-w$, $w'=-z$ (but this is the only other option). Last, for $\nu\geq 2$, we may take $E=(0, \frac{1}{\nu})$ and $E'=(a', a'+\frac{1}{\nu})$, for some $a'\neq 0$.
	
	For two coefficients, apart from the case when  $\ft{\chi_E}(0), \ft{\chi_E}(1)$ are known, thus defining $E$ uniquely, it is also easy to see that knowing $\ft{\chi_E}(1), \ft{\chi_E}(2)$ also determines $E$. No other combination of two Fourier coefficients
	$\ft{\chi_E}(m), \ft{\chi_E}(n)$
	determines the set (we skip the details).
	
	For $N\geq 2$, let us point out that even equality of all Fourier coefficients of two sets
	$$
	E=\bigcup_{j=1}^N (a_j, b_j),\ \ \ E'=\bigcup_{j=1}^N(a_j',b_j')
	$$
	at $1, 2, \ldots 2N-1$, is not enough to conclude $E=E'$, while, by Theorem \ref{th:intervals}, the Fourier coefficients from 0 to $N$ suffice.
	
	To prove the claim, let us begin with a simple observation. Let $x=(x_1, \ldots, x_{2N})$. Then, 
	\begin{equation}\label{vanishing}
		s_1(x)=s_2(x)=\ldots s_{2N-1}(x)=0 \iff \sigma_{1}(x)=\sigma_2(x)=\ldots \sigma_{2N-1}(x)=0.
	\end{equation}
	(We keep the notation for power sums and elementary symmetric functions that was used in the proof of Theorem \ref{th:intervals}.)
	The fact that vanishing power sums give vanishing elementary symmetric functions follows directly from the Newton-Girard formulas \eqref{NG}. For the converse, observe that always $s_1(x)= \sigma_{1}(x)$ and then apply (\ref{NG}) consecutively for $k=2, \dots, 2N-1$ (or observe that the power sums are themselves symmetric functions, hence they can be expressed via the elementary symmetric functions).
	
	Fix some $\theta\in(0, \pi/N)$ and consider two regular $2N$-gons $\mathcal{P}$ and $\mathcal{Q}$ on the unit circle, with vertices arranged counterclockwise, defined by
	$$
	V_{\mathcal{P}} = \Set{z_1, \ldots, z_N, w_1',\ldots, w_N'}
	$$
	and
	$$
	V_{\mathcal{Q}} = \Set{w_1,\ldots, w_N, z_2', \ldots, z_N', z_1'}
	$$
	(see Fig.\ \ref{fig:2-polygons}, where $V_{\mathcal P}$ is blue and $V_{\mathcal Q}$ is red) where
	$$
	z_1 = 1,\ w_1 = e^{i\theta}.
	$$
	\begin{figure}[h]
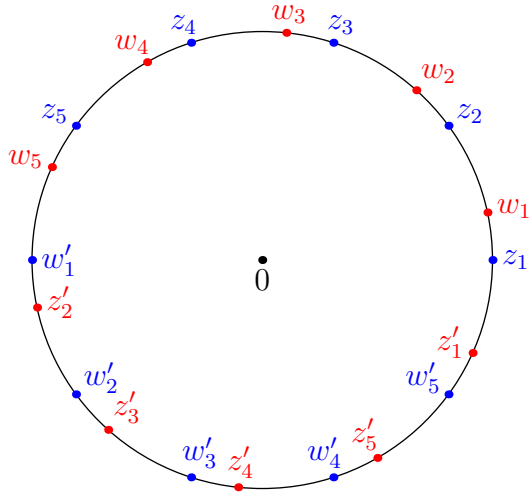

		\begin{center}
			%
			%
			%
			%
				%
			\asyinclude{circle.asy}
		\end{center}
		\caption{One possible selection of arcs $\Ds\bigcup_{j=1}^N(z_j, w_j)$ and $\Ds\bigcup_{j=1}^N(z_j', w_j')$ with the same Fourier coefficients of order $1, 2, \ldots, 2N-1$ (shown for $N=5$).}\label{fig:2-polygons}
	\end{figure}
	It follows that the numbers in $V_{\mathcal P}$ are the roots of the polynomial $z^{2N}-1$ and that the numbers in $V_{\mathcal Q}$ are the roots of the polynomial $z^{2N}-e^{i2N\theta}$. Since the elementary symmetric functions of the roots of a polynomial are the coefficients of the polynomial, it follows that the numbers $\sigma_{\nu}(z,w')$, $\sigma_{\nu}(z',w)$ vanish for $\nu=1, \dots, 2N-1$. By (\ref{vanishing}), we also have $s_{\nu}(z,w')=s_{\nu}(z',w) = 0$ for $\nu=1, \dots, 2N-1$. Then, we have $\ft{\chi_E}(\nu) = \ft{\chi_{E'}}(\nu)$ for all $\nu=1, \dots, 2N-1$, however the sets $E$, $E'$ do not coincide, as implied by the arrangement of $V_{\mathcal{P}}$, $V_{\mathcal{Q}}$.
	
	It is interesting to see that many more examples are possible with the points $z_1, \ldots, z_N, w_1', \ldots, w_N'$ located at the vertices of a regular $2N$-gon and the points $z_1', \ldots, z_N', w_1, \ldots, w_N$ located at the vertices of a rotated regular $2N$-gon, but not necessarily in the order shown in Fig.\ \ref{fig:2-polygons}. As explained above these locations guarantee that the two sets $\bigcup_{j=1}^N(z_j, w_j)$ and $\bigcup_{j=1}^N(z_j', w_j')$ (with an obvious and excusable abuse of notation) have the same Fourier coefficients of order $1, 2, \ldots, 2N-1$ and they are of course not equal.
	
	One needs to find the arrangements of the points $z_j, w_j, z_j', w_j'$ on the vertices of these two polygons so that the following rules are satisfied:
	\begin{enumerate}
		\item
		The points $z_1, \ldots, z_N, w_1, \ldots, w_N$ appear on the circle in the counterclockwise order. Same for the points $z_1', \ldots, z_N', w_1', \ldots, w_N'$. This rule ensures that the arcs $(z_i, w_i)$ are non-overlapping and the same for the arcs $(z_i', w_i')$.
		\item
		The $z_j$ and $w_j'$ are blue (polygon ${\mathcal P}$) and the $z_j'$ and $w_j$ are red (polygon ${\mathcal Q}$).
	\end{enumerate}
	We can enumerate these arrangements by viewing this problem as a variant of the so-called {\it Terquem's problem} (see, e.g., \cite[Problem 30 on p.\ 120, and solution on p.\ 170]{stanley2011enumerative}). Terquem's problem asks in how many ways we can choose a sequence $a_1 < a_2 < \ldots < a_k$ from the set $\Set{1, 2, \ldots, n}$ whose elements alternate between odd and even. By viewing blue as odd and red as even on our polygons we see that we can restate our problem as follows:
	\begin{quotation}
		In how many ways can we select an alternating sequence $a_1 < a_2 < \ldots < a_{2N}$ from the set $\Set{1, 2, \ldots, 4N}$ such that its complement is also alternating.
	\end{quotation}
	The numbers $a_1, \ldots, a_{2N}$ correspond to the choices for the labels $z_1, w_1, \ldots, z_N, w_N$ and the complementary set corresponds to the labels $z_1', w_1', \ldots, z_N', w_N'$.
	
	Following \cite[p.\ 170]{stanley2011enumerative} we can encode the sequence $a_i$ via the sequence $b_i$ defined by
	$$
	b_i = a_i - i + 1.
	$$
	This sequence is increasing
	$$
	b_1 \le b_2 \le \ldots \le b_{2N}
	$$
	and gives back the sequence $a_i$ as $a_i = b_i + i -1$ (which is strictly increasing).
	It also satisfies the bounds
	$$
	1 \le b_i \le 2N+1.
	$$
	The alternating property of the sequence $a_i$ translates exactly to the $b_i$ being all odd. If we did not care about the complement of the sequence $a_i$ being also alternating then, as explained in \cite{stanley2011enumerative}, all we would have to do is select {\it with replacement} the $2N$ numbers $b_i$ among the odd numbers of the set $\Set{1, 2, \ldots, 2N+1}$, that is from $N+1$ numbers. To ensure that the complement is also an alternating sequence it is necessary and sufficient to ensure that the intervals (consecutive values) defined by the sequence $a_i$ are all of {\it even} length. But an interval of the $a_i$ translates into an interval of constancy for the corresponding $b_i$. Summarizing, the $b_i$ must be odd and be selected an even number of times each. To achieve this we select with replacement $N$ numbers from the odd numbers of the set $\Set{1, 2, \ldots, 2N+1}$ and then double the number of times each selection appears. This enumerates the $b_i$ and therefore also the $a_i$. We omit the details.

	\section{Point masses on $\TT^d$}\label{sec:2d}

	\subsection{Point masses on $\TT$}\label{sectonetorus}
	The one-dimensional problem has a very long history, going back to Gaspard de Prony's work \cite{de1795essai} from 1795. Since then, many solutions have been proposed, like Pisarenko's method \cite{pisarenko1973retrieval}, MUSIC \cite{schmidt1986multiple} or ESPRIT \cite{roy1989esprit}. Still, there is ongoing research on further improvements, see \cite{derevianko2021esprit} for a recent approach.
	
	We show here another approach from the Electrical Engineering literature (see e.g.\ \cite{vetterli2002sampling}) with so-called annihilation filters. We will make use of Theorem \ref{onetorus} repeatedly when solving the same problem on $\TT^d$.
	
	\begin{theorem}\label{onetorus}
		Suppose $\mu$ is a measure on $\TT$ which is a sum of at most $N$ complex point masses.
		Then $\mu$ is determined by the data
		\beql{prony-data}
		\ft{\mu}(j),\ \ \ j=-N+1, -N+2, \ldots, N.
		\eeq
	\end{theorem}
	
	\begin{proof}
		Suppose that $\mu=\sum_{j=1}^K c_j \delta_{\theta_j}$, with $c_j \in \CC\setminus\Set{0}$, the
		$\theta_j$ all different and $K \le N$.
		It follows that
		\beql{lin-comb}
		\ft{\mu}(n) = \sum_{j=1}^K c_j \rho_j^{-n},\ \ \ (\rho_j=e^{2\pi i \theta_j},\ n \in \ZZ).
		\eeq
		Define the polynomial
		$$
		a(z) = \prod_{j=1}^K (z-\rho_j) = \ft{a}(0) + \ft{a}(1)z + \ldots + \ft{a}(K-1)z^{K-1}+z^K,
		$$
		and note that $\ft{a}(n)$ also denotes its Fourier coefficients when viewed as a function on $\TT$.
		Since $a\mu = 0$ it follows that
		\beql{conv-infinite}
		\ft{a} * \ft{\mu}(n) = 0,\ \ \ (n \in \ZZ).
		\eeq
		We now show that the polynomial $a(z)$ is determined up to constant multiples by the conditions
		\begin{eqnarray}
			\deg{a} &\le& K \label{degree}\\
			\ft{a} * \ft{\mu}(n) &=& 0,\ \ \ (n \in [1, K]).\label{conv}
		\end{eqnarray}
		It is enough to show that \eqref{degree} and \eqref{conv} together imply 
		that $a(z)$ vanishes on the $\rho_j$ (this is the same as \eqref{conv-infinite}),
		whose number is $K$, the same as the degree of $a$.
		
		Substituting \eqref{lin-comb} in \eqref{conv} we get for $l=1,2,\ldots,K$
		\begin{eqnarray*}
			0 &=& \sum_{k=0}^K \ft{a}(k) \ft{\mu}(l-k),\\
			&=& \sum_{k=0}^K \ft{a}(k) \sum_{j=1}^K c_j \rho_j^{k-l},\\
			&=& \sum_{j=1}^K c_j \rho_j^{-l} \sum_{k=0}^K \ft{a}(k) \rho_j^k,\\
			&=& \sum_{j=1}^K c_j \rho_j^{-l} a(\rho_j). 
		\end{eqnarray*}
		Observe that the $K\times K$ Vandermonde matrix $\rho_j^{-\nu}$,
		$j,\nu=1,\ldots,K$, is nonsingular and,
		therefore, all $c_j a(\rho_j)$ are 0, $j=1,2,\ldots,K$. Since all $c_j$ are nonzero this implies that
		$$
		a(\rho_j)=0,\ \ \ j=1,2,\ldots,K.
		$$
		
		Suppose $\mu_1, \mu_2$ are two measures with the same Fourier data \eqref{prony-data}:
		$$
		\mu_1 = \sum_{j=1}^{K_1} c_{1,j} \delta_{\theta_{1,j}},\ \ 
		\mu_2 = \sum_{j=1}^{K_2} c_{2,j} \delta_{\theta_{2,j}},
		\ K_1 \le K_2 \le N,
		$$
		and
		\beql{agreement}
		\ft{\mu_1}(n) = \ft{\mu_2}(n),\ \ (n=-N+1,\ldots,N-1,N).
		\eeq
		We are assuming that all $\theta_{1,j}$ are distinct and so are all $\theta_{2,j}$, and
		that $c_{1,j}, c_{2,j} \in \CC\setminus\Set{0}$.
		Write also $\rho_{1,j}=e^{2\pi i \theta_{1,j}}$ and
		$\rho_{2,j}=e^{2\pi i \theta_{2,j}}$.
		
		Write
		$$
		a_1(z) = \prod_{j=1}^{K_1} (z-\rho_{1,j}),\ \ a_2(z) = \prod_{j=1}^{K_2} (z-\rho_{2,j}).
		$$
		We have $\ft{a_1}*\ft{\mu_1}(n) = \ft{a_2}*\ft{\mu_2}(n) = 0$ for all $n$ but we also
		have $\ft{a_1}*\ft{\mu_2}(n) = 0$ for $n=1,\ldots,N$, because of \eqref{agreement}.
		Applying the fact that $a(z)$ is determined by \eqref{degree} and \eqref{conv} with $K=K_2$ we obtain
		that $a_1(z)=a_2(z)$ hence $K_1=K_2$ and $\Set{\rho_{1,j}} = \Set{\rho_{2,j}}$.
		
		It remains to show that the linear map
		$$
		(c_1,\ldots,c_K) \to (\ft{\mu}(1),\ldots,\ft{\mu}(K))
		$$
		is injective. This map is given by \eqref{lin-comb} and it is easily seen to be nonsingular as its
		determinant is a multiple of the Vandermonde determinant.
	\end{proof}
	
	\begin{remark}
	The number of samples is sharp, as we recover the $2N$ parameters using $2N$ samples. However, note that we use complex samples, to recover $N$ complex parameters and $N$ parameters in $\TT$. One can easily check that the proof extends to the case of point measures on $\TT + i\RR$, where \eqref{lin-comb} becomes
	$$
			\ft{\mu}(n) = \sum_{j=1}^K c_j \rho_j^{-n},\ \ \ (\rho_j=e^{2\pi i (\theta_j+i\xi_j)},\ n \in \ZZ).
	$$
	In case of real coefficients and point measures on $\TT$ one can utilize $\overline{\hat \mu(-k)} = \hat \mu(k)$ to use only $\hat\mu(k),~k=0,\ldots, N$. This observation is an important part of the unitary ESPRIT algorithm \cite{haardt1995unitary}, which uses only real-valued computation to solve the problem.
	\end{remark}
	
	\subsection{Connections to interpolation}\label{sec:interpolation}
	
	Suppose $\mu = \sum_{j=1}^N c_j \delta_{u_j}$, where $u_j =(u_{j1},\ldots,u_{jd}) \in \TT^d$ are distinct points, all $c_j$ are non-zero and $d>1$.
	Can we recover $\mu$ from a number of Fourier coefficients that is close to the number of degrees of
	freedom (in this case $(d+1)N$ or $(d+2)N$ depending on whether $c_j \in \RR$ or $c_j \in \CC$)?
	
	Here the existing results do not seem to be final. In the special case where all $u_{j1}$ are different (or equivalently all $u_{jk}$ for a fixed $k=1,\ldots,d$) the problem is solved using the one-dimensional theory with $O(dN)$ Fourier coefficients, namely by using the Fourier coefficients at the locations $(m, \epsilon_2,\dots,\epsilon_d)$ with $m=0, 1, \ldots, N$ and $\epsilon_j\in\Set{0,1}$, assuming all $u_{j1}$ are different (we generalize this in Theorem \ref{th:max-k}).
	In the general case and without imposing any restrictions on the locations $u_j$, it has only been known until this work how to recover $\mu$ using $O(N^2)$ Fourier coefficients \cite[Section \MakeUppercase{\romannumeral 3}.C]{maravic2004exact} in the two dimensional case and using $O(N^2\log^{2d-2}N)$ coefficients in the general case \cite{sauer2018prony}.
	
	If one allows for the collection of Fourier coefficients used to depend on the data then one can recover the parameters with $O(N)$ Fourier coefficients, see \cite{plonka2013many} for the case $d=2$ and \cite{cuyt2018multivariate} for arbitrary $d$. It was conjectured in \cite{plonka2013many} that recovery of the parameters in this problem with $d=2$ is always possible with $O(N)$ Fourier coefficients which do not depend on the data and are on four predetermined lines. This conjecture was disproved in \cite{diederichs2015parameter} but the possibility remains that some more general set of $O(N)$ Fourier coefficients suffices. (In Theorem \ref{th:NlogN} we show that $O(N \log N)$ Fourier coefficients suffice.) For the general $d$ dimensional case it was shown in \cite{griesmaier2017multifrequency} that taking a total of $O(N^2)$ samples on scattered line allows for a reconstruction.
	
	\begin{definition}
		Suppose $\Omega \subseteq \ZZ^d$ and $k=1,2,3,\ldots$. We call $\Omega$
		\underline{$k$-interpolating} if
		whenever $u_1,\ldots,u_{\ell} \in \TT^d$, $\ell\le k$, are distinct and $d_1, \ldots, d_{\ell} \in \CC$ we can find coefficients
		$c_\omega$, $\omega\in\Omega$, such that
		$$
		d_j = \sum_{\omega\in\Omega} c_\omega e^{2\pi i \omega\cdot u_j},\ \ j=1,2,\ldots,\ell.
		$$
		We call $\Omega$ \underline{$k$-sufficient}
		if we can recover any measure $\mu = \sum_{j=1}^{\ell} c_j \delta_{u_j}$, $\ell\le k$,
		(with unknown $c_j\in\CC\setminus\Set{0}$, unknown and distinct $u_j\in\TT^d$) from its Fourier coefficients at $\Omega$
		$$
		\ft{\mu}(\omega) = \sum_{j=1}^\ell c_j e^{-2\pi i \omega\cdot u_j},\ \ \omega\in\Omega.
		$$
	\end{definition}
	The connection between the concepts of $k$-interpolation and $k$-sufficiency is the following.
	\begin{theorem}\label{th:connection}
		\beql{implications}
		\Omega\mbox{ is $(2N)$-interpolating } \Longrightarrow
		\Omega\mbox{ is $N$-sufficient } \Longrightarrow
		\Omega\mbox{ is $N$-interpolating.}
		\eeq
	\end{theorem}
	\begin{proof}
		To prove Theorem \ref{th:connection} let us first make the
		following remark, which says that if we can solve the problem of sufficiency
		with the locations fixed then
		we can also solve the problem with unknown (but fewer) locations.
		\begin{lemma}\label{lm:fix-the-locations}
			Suppose $\Omega \subseteq \ZZ^d$ is such that the mapping
			\beql{map}
			\mu \to (\ft{\mu}(\omega),\ \omega\in\Omega)
			\eeq
			is injective on the set
			\beql{anchored-set}
			\Set{\mu = \sum_{j=1}^{k} c_j \delta_{u_j}:\ c_j \in \CC},
			\eeq
			for any choice of $k \le 2N$ and distinct points $u_1,\ldots,u_k \in \TT^d$.
			Then the mapping \eqref{map} is injective also on the set
			\beql{free-set}
			\Set{\mu = \sum_{j=1}^{\ell} c_j \delta_{u_j}:\ \ell\le N, c_j \in \CC, u_j \in \TT^d, \mbox{ $c_j\neq 0$ and the $u_j$ are distinct}}.
			\eeq
		\end{lemma}
		\begin{proof}
			Suppose $\mu=\sum_{j=1}^{\ell_1} c_j \delta_{u_j}$ and $\nu =\sum_{j=1}^{\ell_2} d_j \delta_{v_j}$,
			$\ell_1, \ell_2 \le N$,
			are two different measures with the same image under \eqref{map}.
			Then $\mu-\nu$ is a non-zero measure supported on $\le 2N$ points and is mapped to $0$
			under \eqref{map}.
			This conflicts with the injectivity of \eqref{map} on the set \eqref{anchored-set}.
			
		\end{proof}
		
		For the mapping \eqref{map} on the set \eqref{anchored-set} to be injective it is necessary and sufficient that the matrix
		\beql{matrix}
		\left( e^{2\pi i u_j\cdot \omega} \right)_{j=1,\ldots,k,\ \omega\in\Omega}
		\eeq
		has rank $k$.
		Write $e_\omega(x) = e^{2\pi i \omega\cdot x}$.
		\begin{lemma}\label{lm:transpose}
			The mapping \eqref{map} on the set \eqref{anchored-set}
			is injective if and only if the functions $e_\omega$ form an interpolating
			set for the set $\Set{u_1,u_2,\ldots,u_k}$,
			i.e.\ for any values $d_j \in \CC$ there is a $\CC$-linear combination of the $e_\omega$ which 
			takes the value $d_j$ at $u_j$.
		\end{lemma}
		\begin{proof}
			The $e_\omega$ are interpolating at the $u_j$ if and only if the rank of the 
			matrix in \eqref{matrix} is $k$. By the preceding remark this is equivalent to the mapping \eqref{map} on the set \eqref{anchored-set} being injective.
			
		\end{proof}
		
		Let us complete the proof of Theorem \ref{th:connection}. If $\Omega$ is $(2N)$-interpolating
		and $u_1,\ldots,u_{2N}$ are distinct points in $\TT^d$
		it follows from Lemma \ref{lm:transpose} that the mapping $\mu\to\ft{\mu}|_\Omega$ is injective
		on the set $\eqref{anchored-set}$ (with $k=2N$). From Lemma \ref{lm:fix-the-locations} it follows that
		it is also injective on the set \eqref{free-set}, hence $\Omega$ is $N$-sufficient.
		
		If $\Omega$ is $N$-sufficient and $u_1,\ldots,u_k \in \TT^d$, $k\le N$, are distinct points then the matrix
		\eqref{matrix} has rank $k$. Therefore the mapping \eqref{map} is injective (with the $u_j$ fixed)
		and from Lemma \ref{lm:transpose} we get that the functions $e_\omega(x)$ are interpolating,
		which is what it means for $\Omega$ to be $N$-interpolating.
		
	\end{proof}
	
	Now it becomes clear that the situation in dimension 2 is significantly harder than in dimension 1.
	The reason is that interpolation is harder. Indeed, in dimension 1 one can easily find a set
	of $N$ functions the linear combinations of which can interpolate any data on any $N$ points. One such example
	of functions are the monomials $1, x, x^2, \ldots, x^{N-1}$ and another example
	are the functions $1, e^{2\pi ix}, e^{2\pi i 2x}, \ldots, e^{2\pi i (N-1)x}$ (when all $u_j \in \TT$).
	
	Such a set of functions is called a {\em Chebyshev or Haar system} and it is
	well known and easy to prove that {\em continuous} Chebyshev systems
	do not exist except in dimension 1 \cite{mairhuber1956haar}. Indeed, suppose that $S \subseteq \RR^2$ is an open set and
	the continuous functions $f_j:S\to\RR$, $j=1,2,\ldots,N$,
	are such that for any set of $N$ distinct points $u_j \in S$
	we can find a linear combination of the $f_j$ which interpolates any given \underline{real} data at the $u_j$.
	This means that for any choice of the distinct points $u_j$ the determinant of the matrix $f_i(u_j)$, $i,j=1,2,\ldots,N$,
	is non-zero. Choose then the $u_j$ to belong to an open disk in $S$ and carry out a continuous movement
	of the points $u_1$ and $u_2$ so that they do not collide between themselves and with any of the other points
	and such that, at the end of the motion, the two points have exchanged their positions. The determinant
	of the matrix has changed sign and, since it has varied continuously during the motion, it follows that the
	determinant has vanished at some point during the exchange, a contradiction.
	It is proved in \cite{mairhuber1956haar} that the existence of a (continuous) Chebyshev system on a set
	$S \subseteq \RR^d$ is only possible when $S$ is homeomorphic to a closed subset of a circle.
	
	This argument is strictly for the real case of course but it has been extended \cite{schoenberg1961unicity}
	to the case of complex functions:
	there is a complex Chebyshev system for domains in $\CC$ but not for domains in $\CC^2$ or in higher dimension.
	More specifically, it is proved in \cite{schoenberg1961unicity,henderson1973haar}
	that a \underline{complex} continuous Chebyshev
	system exists on a locally connected set $S$ if and only if $S$ is homeomorphic to a closed subset of $\RR^2$.
	This result allows us to prove that the situation in $\TT^2$ is strictly worse than in $\TT$.
	\begin{theorem}\label{th:strict-inequality}
		Let $d>1$. Suppose $\Omega\subseteq\ZZ^d$ is $N$-sufficient. Then $\Abs{\Omega} > N$.
	\end{theorem}
	\begin{proof}
		If such an $\Omega$ had size $N$ then, according to Theorem \ref{th:connection}, the corresponding
		set of characters $e_\omega(x) = e^{2\pi i \omega\cdot x}$, $\omega\in\Omega$, would be a continuous Chebyshev system on $\TT^d$,
		According to \cite{schoenberg1961unicity,henderson1973haar} this would make $\TT^d$ embeddable into the plane,
		which it is not.
		
	\end{proof}
	
	Again using the connection to interpolation let us now give a new proof, different from the one given in \cite[Section \MakeUppercase{\romannumeral 3}.C]{maravic2004exact} for the case $d=2$ and \cite{kunis2016multivariate, sauer2017prony} for general $d$, of the following fact.
	\begin{theorem}\label{thm:ndsufficient}
		There is $\Omega \subseteq \ZZ^d$ of size $O(N^d)$ which is $N$-sufficient.
	\end{theorem}
	\begin{proof}
		It is enough to produce a set $\Omega\subseteq\ZZ^d$ of size $O(N^d)$ such that the set of corresponding
		exponentials $e^{2\pi i \omega\cdot x}$ is $2N$-interpolating, i.e. its linear combinations can interpolate
		any values at any $2N$ distinct points in $\TT^d$.
		We use the fact that for any set of $2N$ distinct points in $\CC^d$ and any complex data
		on them there is a complex
		polynomial in $d$ variables of degree at most $2N-1$ (the degree of each monomial is the sum of the exponents  of the variables) which interpolates the data.
		To see this observe that for any set of $2N$ distinct points $x_1, \ldots, x_{2N} \in \CC^d$ we can find a vector $u \in \CC^d$ such that the complex numbers $t_i = v\cdot x_i$ are all different. Let now $p$ be a one-variable polynomial of degree $\le 2N-1$ which interpolates the given data on the points $t_i$. Then $q(x) = p(v\cdot x)$ is a two-variable polynomial that interpolates the given data on the points $x_i$. The degree of $q$ is no larger than $2N-1$.
		
		Take $\Omega=\Set{m\in\NN_0^d: \sum_{j=1}^d m_j\le 2N-1}$.
		Suppose
		$
		u_1,\ldots,u_{2N}\in \TT^d
		$
		are distinct and $d_1,\ldots,d_{2N} \in \CC$.
		Let $p(z)=\sum_{m}p_{m}z^m$ have degree at most $2N-1$ and interpolate the data $d_k$ at the points
		$(e^{2\pi i u_{j1}}, \ldots, e^{2\pi i u_{jd}}) \in \CC^d$, $j=1,2,\ldots,2N$ (note that these are distinct points as
		the $u_d$ are in $\TT^d$ not in $\RR^d$).
		We have
		$$
		d_j = p(e^{2\pi i u_{j1}}, \ldots, e^{2\pi i u_{jd}}) = \sum_{m\in\Omega} p_{m} e^{2\pi i u_j\cdot m},
		$$
		which means that the functions $e^{2\pi i \omega\cdot x}$, $\omega\in\Omega$, are interpolating
		the arbitrary data $d_j$ at the $2N$ arbitrary points $u_j$, as we had to prove.
		
	\end{proof}
	
	\subsection{Small sufficient sets for $\TT^d$}\label{sec:t2}
	
	Next we provide a case where the sufficient number of coefficients for $N$ points of $\mathbb{T}^d$ is $O(k^{d-1}N)$, for some $1\leq k\leq N$.
	
	\begin{theorem}\label{th:max-k}
		Let $\mu = \sum_{j=1}^N c_j \delta_{u_j}$, where $u_j =(x_j,y_j) \in \TT\times \TT^{d-1}$ are distinct points and all $c_j \in \mathbb{C}\backslash \{0\}$. Assume that the number of points $u_j$ that share the same $x$ coordinate is at most $k$, for some $1\leq k\leq N$. Then $\mu$ can be recovered by $O(k^{d-1}N)$ Fourier coefficients.
	\end{theorem}
	
	\begin{proof}
		Write $U = \Set{u_j:\;j=1,\ldots,N}$ and $X = \Set{x_j:\; j=1,\ldots,N}$ for the set of \textit{distinct} $x$ that appear as first coordinates for the points in $U$. Notice that $X$ may have fewer than $N$ points.
		
		Recall that
		\begin{align*}
			\ft{\mu}(m, n) &= \sum\limits_{j=1}^{N}c_je^{-2\pi i (mx_j+n\cdot y_j)}\\ 
			&= \sum_{x \in X} \left( \sum_{y: \; (x, y) \in U} c_{(x, y)} e^{-2\pi i n \cdot y} \right) e^{-2\pi i m x},
			\;m\in \mathbb{Z}, \, n\in \ZZ^{d-1}.
		\end{align*}
		For fixed $n$ the numbers $\ft{\mu}(m, n)$ are the Fourier coefficients of a collection of point masses at the points of $X$ (some of these point masses may be 0).
		
		Consider the data
		$$
		\ft{\mu}(-N, \ell),\; \ldots,\; \ft{\mu}(0, \ell),\; \ft{\mu}(1, \ell),\; \ldots,\; \ft{\mu}(N, \ell),
		$$
		where $\ell\in\ZZ^{d-1}$ is fixed.	By Theorem \ref{onetorus}, we can recover the sums
		\begin{equation}\label{sumsy}
			S(x, \ell) = \sum\limits_{y: \; (x, y) \in U} c_{(x, y)} e^{-2\pi i\ell \cdot y}, \quad \ell\in \ZZ^{d-1}.
		\end{equation}
		and the corresponding $x \in X$. Notice that we only ``see'' the $x$ for which $S(x, \ell) \neq 0$.

		For fixed $x \in X$, define the measure on $\TT^{d-1}$
		$$
		\lambda_x = \sum\limits_{y: \; (x, y) \in U} c_{(x, y)} \delta_{y},
		$$
		which is supported on at most $k$ locations on $\TT^{d-1}$. By Theorem \ref{thm:ndsufficient}, there is $\Omega\subset \ZZ^{d-1}$ of size at most $O(k^{d-1})$ that is $k$-sufficient. Thus, knowing the Fourier coefficients of $\lambda_x$ at $\ell\in\Omega$ is sufficient to recover the measure. Knowing these Fourier coefficients means precisely knowing the sums in (\ref{sumsy}), so we recover the points $y$ sitting over each $x \in X$ and the corresponding coefficients. The proof of the Theorem is complete.
		
	\end{proof}
	
	Finally we come to the main result of this section. One can view Theorem \ref{th:NlogN} as a more sophisticated version of Theorem \ref{th:max-k}, where the gain comes from being able to distinguish which $x \in \TT$ have many points projected onto them. This set of $x$ cannot be large.
	
	This theorem was first proved in \cite{sauer2018prony}, using techniques from computational algebra. We give an elementary proof.
	
	\begin{theorem}\label{th:NlogN}
		There is $\Omega \subseteq \ZZ^d$ of size $\Abs{\Omega} \leq C_d N \log^{d-1} (N)$ which is $N$-interpolating.
		One such set is the positive octant of the hyperbolic cross
		$$
		\Gamma_N^d = \left\{ n\in\NN_0^d ~:~\prod_{j=1}^d\left( n_j + 1 \right) \leq N\right\}.
		$$
		
	\end{theorem}
	
	\begin{proof}
		Let $U=\Set{u_j:\; j=1,\ldots, N}$. As noted in \eqref{matrix}, it suffices to show that the vectors
		$$
		v_\Gamma(u_j) = (e^{2\pi i u_j\cdot w})_{\omega\in \Gamma_N^d},\quad j=1,\ldots,N
		$$
		are linearly independent. We use induction in $d$. For the base case $d=1$, observe that the matrix $(e^{2\pi i u_j\omega})_{j=1, \ldots, N, \; \omega=0, \ldots, N-1}$ is Vandermonde. For $d\geq 2$, assume on the contrary that there are $c_u\in\CC\setminus\Set{0},\;u\in U$ (we can exclude nodes with zero coefficients) satisfying
\begin{equation*}
	\sum_{u\in U} c_{u} v_\Gamma(u) = 0.
\end{equation*}
		Let $X=\Set{x_j\in \TT^{d-1}:\; j=1,\ldots,N}$ be the set of distinct $x$ that appear as the first $d-1$ entries of points in $U$ (again, $X$ may have fewer than $N$ points). As in the proof of Theorem \ref{th:max-k}, we note that the condition of linear dependence rewrites as
		\begin{equation}\label{eq_decomp}
			\sum_{j=1}^N c_{u_j} e^{-2\pi i (m\cdot x_j+ky_j)} =  \sum_{x \in X} \left( \sum_{y: \; (x, y) \in U} c_{(x, y)} e^{-2\pi i k y} \right) e^{-2\pi i m \cdot x}=0,
			\quad m\in \ZZ^{d-1}, k\in\ZZ, \;(m,k)\in \Gamma_N^{d}.
		\end{equation} 
	Observe that we have
		\begin{equation}\label{eq:dataGamma}
			(m, 0) \in \Gamma_N^d \; \text{ for } m\in\Gamma_{N}^{d-1} \quad \text{ and } \quad	(m, k-1) \in \Gamma_N^d \; \text{ for } m\in\Gamma_{\lfloor N/k\rfloor}^{d-1}, \; k\in \NN.
	\end{equation}
		Then, as (\ref{eq_decomp}) holds for $(m, 0) \in \Gamma_N^d$ for all $m\in\Gamma_N^{d-1}$, we can use the induction hypothesis to conclude that
		\begin{equation}\label{eq_coef0}
			\sum_{y: \; (x, y) \in U} c_{(x, y)} e^{-2\pi i yk } = 0,\quad \text{for } k=0, \; \text{for all } x\in X.
		\end{equation} 
		We partition $X$ according to how many points of $U$ project to each point:
		\begin{equation*} 
			X = X_1 \sqcup \ldots \sqcup X_r\ \ \ (r \le N),
		\end{equation*}
		where $X_t = \Set{x \in X:\ \Abs{\Set{y: (x, y) \in U}} = t}$. By \eqref{eq_coef0}, we see that
		$$
		c_{(x,y)} = 0 \quad \text{for all } x\in X_1, \; (x,y)\in U,
		$$
		which contradicts $c_u\in\CC\setminus\Set{0}$ and thus $X_1=\emptyset$. Now we use the crucial observation that
		\begin{equation}\label{crucialobs}
			|U| = \sum_{j=1}^r j|X_j|
		\end{equation}
		which implies with $X_1=\emptyset$ that $|X|\leq \lfloor N/2\rfloor$.
	
		Taking data at \eqref{eq:dataGamma} for $k=0, 1$,  by the induction hypothesis we see that \eqref{eq_coef0} holds true for $k=0,1$. If $X_2$ is not empty, then, for any $x\in X_2$, the summands in \eqref{eq_coef0} are exactly $2$, for each $k\in \{0,1\}$.  Thus, we have a homogeneous $2\times 2$ system, with a Vandermonde matrix of coefficients, so 
		$$
		c_{(x,y)} = 0 \quad \text{for all } x\in X_2, \; (x,y)\in U,
		$$ a contradiction. That allows us to deduce $X_2=\emptyset$, giving us $|X|\leq \lfloor N/3\rfloor$. Repeating the argument $r$ times results in the contradiction $X=\emptyset$.
	\end{proof}
	
	We use Theorem \ref{th:connection} to see that $O(N\log^{d-1} N)$ samples are sufficient for unique determination of a point measure $\mu$ of at most $N$ peaks.
	
	\begin{corollary} \label{cor:ddimsufficient}
		There is $\Omega \subseteq \ZZ^d$ of size $\Abs{\Omega} = O( N \log^{d-1} (N))$ which is $N$-sufficient. One such set is $\Gamma_{2N}^d$.
	\end{corollary}
	
Note, however, that the proof cannot be converted in an algorithm recovering the measure from its Fourier samples. Next, we show that such an algorithm exists for a slightly larger sampling set.
	
	\begin{theorem} \label{thm:ONreco}
		Any measure $\mu = \sum_{j=1}^K c_j\delta_{u_j},~K\leq N$, where $u_j=(x_j,y_j)\in\TT^{d-1}\times \TT$ are distinct points and $c_j\in\CC\setminus\Set{0}$, is determined by its Fourier coefficient on the set
		\begin{equation}
			\tilde \Gamma_N^d = \left\{ n\in\NN_0^d ~:~\prod_{j=1}^d \left\lceil \frac{n_j+1} 2 \right\rceil \leq N\right\}. \label{cross2}
		\end{equation} 
	\end{theorem}
	
	\begin{proof}
		The proof works by using one-dimensional methods to give a large set of candidates. These candidates are all point measures with at most $N$ summands. By Corollary \ref{cor:ddimsufficient}, only one can fit all the available data, as $\Gamma_{2N}^d\subset\tilde\Gamma_N^d$ . We use the same notation as in the proof of Theorem \ref{th:NlogN}.
		
		Again, we use induction in $d$, where the case $d=1$ is a consequence of Theorem \ref{onetorus}. Further, note that for $r\in\NN\setminus\Set{0}$ we have 
		\begin{equation}\label{eq_slicetildegamma}
			(m, k)\in \tilde \Gamma_N^d \text{ for } m\in\tilde\Gamma_{\lfloor N/r\rfloor}^{d-1} \text { and } k=2r-2, \; 2r-1.
		\end{equation}
		We again use the decomposition \eqref{eq_decomp} and introduce the notation
		$$
		c_x(k) = \sum_{y: \; (x, y) \in U} c_{(x, y)} e^{-2\pi i yk },
		$$
		to write $$\ft{\mu}(m, k)=\sum_{x \in X} c_x(k)e^{-2\pi i m \cdot x}, \quad m\in \ZZ^{d-1}, \; k\in \ZZ.$$
		Applying our algorithm for $d-1$ to the samples $\ft{\mu}(m, 0)$, where $m\in\tilde\Gamma_N^{d-1}$, we determine the quantities 
		$$
		c_x(0) = \sum_{y:\; (x, y) \in U} c_{(x, y)}
		$$
		as well as all $x \in X$ for which $c_x(0) \neq 0$. This includes all $x \in X_1$.
		Again from our algorithm for $d-1$, the numbers $\ft{\mu}(m, 1)$, $m\in\tilde\Gamma_N^{d-1}$, determine the quantities
		$$
		c_x(1) = \sum_{y:\; (x, y) \in U} c_{(x, y)} e^{-2\pi i y},
		$$
		and all $x \in X$ for which $c_x(1) \neq 0$. 
		So, by the the data \eqref{cross2} corresponding to $n=(m, k)$, with $m\in\tilde\Gamma_N^{d-1}$ and $k=0, 1$, we can determine, for each point $x$, the Fourier coefficients
		$$
		\ft{\lambda_x}(0),\ \  \ft{\lambda_x}(1)
		$$
		of the one-dimensional measure sitting over $x$:
		$$
		\lambda_x = \sum_{y:\; (x, y) \in U} c_{(x, y)} \delta_y.
		$$
		We also determine those $x \in X$ for which at least one of the numbers
		$$
		c_x(0),\; c_x(1)
		$$
		is non-zero. We collect them in the set $Z_1$. As we said above this includes all $x \in X_1$:
		$$
		X_1\subset Z_1\subset X.
		$$
		
		No we assume for the moment that we could identify the subset $X_1$ in $Z_1$. Then we could determine the part of $\mu$ sitting over $X_1$:
		$$
		\mu_1 = \sum_{x \in X_1} \sum_{y:\;(x, y) \in U} c_{(x, y)} \delta_{(x, y)}.
		$$
		Subtracting $\ft{\mu_1}(m, k)$ from $\ft{\mu}(m, k)$ we see that we know the Fourier coefficients of the measure $\mu-\mu_1$ for all indices in \eqref{cross2}.
		
		The next stage is to determine $\mu_2$, the part of $\mu$ sitting over $X_2$, the points on the $x$-axis with two point masses over them. We will do this using the data
		\beql{reduced-data}
		\ft{\mu-\mu_1}(m, k),\ \  \text{ for }(m, k) \in \tilde\Gamma_{N}^{d-1}\times \{0,1\} \text{ and }(m,k)\in \tilde\Gamma^{d-1}_{\Floor{\frac{N}{2}}}\times \{ 2,3\}.
		\eeq
		From \eqref{crucialobs} the measure $\mu-\mu_1$ contains at most $N/2$ point masses, therefore, by the induction hypothesis, the data \eqref{reduced-data} are now enough to determine the quantities
		$$
		c_x(k) = \sum_{y:\;(x, y) \in U} c_{(x, y)} e^{-2\pi i k y}, \quad x \in X_2 \sqcup X_3 \sqcup \cdots \sqcup X_r, \; k = 0, \ldots, 3,
		$$
		and those $x \in X$ for which at least one of the numbers $c_x(k)$, $k=0, \ldots, 3$, is non-zero, which we collect in the set $Z_2$. This includes all $x \in X_2$ by using the induction hypothesis. Again, assume we were somehow able to identify $X_2$ from the larger set $Z_2$.
		
		If $x \in X_2$ this information suffices, because of Theorem \ref{onetorus}, to determine $\lambda_x$, that is, the part of measure $\mu$ sitting over $x$. So the data $(m, k)\in \NN_0^{d-1}\times \NN_0$ in \eqref{cross2} with $k \le 3$ determine $\mu_2$, the part of measure $\mu$ sitting over $X_2$.
		
		This process continues. The next step is to find, using the data
		$$
		(\mu-\mu_1-\mu_2)^\wedge (m, k),\ \ \text{ where }(m, k)\in \NN_0^{d-1}\times \NN_0 \text{ as in } \eqref{cross2} \text{ with } k \le 5
		$$
		the measure $\mu_3$, the part of $\mu$ sitting over $X_3$. This is again possible since $\mu-\mu_1-\mu_2$ contains at most $N/3$ point masses.
		
		Continuing like this we determine the measure $\mu$ completely.
		
		However, we do not know which subset of $Z_1$ is $X_1$. Instead, we run the whole procedure for every possible choice, not only for $X_1$ but also $X_2\subset Z_2$ and so on. As all sets are finite and only one solution exists, this procedure will recover $\mu$ in a finite number of operations.
	\end{proof}
	
	Theorem \ref{thm:ONreco} is a new result, showing that in principle $O(N\log^{d-1} N)$ samples enable not only to conclude uniqueness (which allows to search the continuous parameter space to recover $\mu$), but to recover the measure using a finite number of computational steps. It was sketched in the PhD thesis \cite{diederichs2018sparse} of the first-named author. However, the algorithm does not have polynomial runtime. It would be interesting to examine whether this is a conceptual barrier or whether more efficient methods exist. The result should be contrasted with the result of Sauer \cite{sauer2018prony}, which uses $O(N^2\log^{2d-2} N)$ samples but has polynomial runtime. 
	
	
	\bibliographystyle{alpha}
	\bibliography{interp}
	
\end{document}